\documentclass[oneside,english]{amsart}
\usepackage[T1]{fontenc}
\usepackage[latin9]{inputenc}
\usepackage{babel}
\usepackage{amstext}
\usepackage{amsthm}
\usepackage{amssymb}
\usepackage[unicode=true]
 {hyperref}

\makeatletter

\newcommand{\noun}[1]{\textsc{#1}}
\let\SF@@footnote\footnote
\def\footnote{\ifx\protect\@typeset@protect
    \expandafter\SF@@footnote
  \else
    \expandafter\SF@gobble@opt
  \fi
}
\expandafter\def\csname SF@gobble@opt \endcsname{\@ifnextchar[%]
  \SF@gobble@twobracket
  \@gobble
}
\edef\SF@gobble@opt{\noexpand\protect
  \expandafter\noexpand\csname SF@gobble@opt \endcsname}
\def\SF@gobble@twobracket[#1]#2{}
\providecommand{\tabularnewline}{\\}

\numberwithin{equation}{section}
\numberwithin{figure}{section}
\theoremstyle{plain}
\newtheorem{thm}{\protect\theoremname}
\theoremstyle{plain}
\newtheorem{lem}[thm]{\protect\lemmaname}
\theoremstyle{plain}
\newtheorem{cor}[thm]{\protect\corollaryname}
\theoremstyle{definition}
\newtheorem{example}[thm]{\protect\examplename}
\theoremstyle{definition}
\newtheorem{defn}[thm]{\protect\definitionname}
\theoremstyle{remark}
\newtheorem{rem}[thm]{\protect\remarkname}

\makeatother

\providecommand{\corollaryname}{Corollary}
\providecommand{\definitionname}{Definition}
\providecommand{\examplename}{Example}
\providecommand{\lemmaname}{Lemma}
\providecommand{\remarkname}{Remark}
\providecommand{\theoremname}{Theorem}

\begin{document}
\title[on simplicity in $\mathsf{Conv}$]{When is a category of adherence-determined convergences simple?}
\author{Frédéric Mynard and \noun{Jerzy Wojciechowski}}
\address{NJCU, Department of Mathematics, 2039 Kennedy Blvd, Jersey City, NJ
07305, USA}
\address{West Virginia University, Department of Mathematics, 94 Beechurst
Ave, Morgantown, WV 26506-6310, USA}
\email{fmynard@njcu.edu}
\email{jerzy@math.wvu.edu}

\maketitle
\global\long\def\G{\mathcal{G}}%
 
\global\long\def\F{\mathcal{F}}%
 
\global\long\def\H{\mathcal{H}}%
\global\long\def\Z{\mathcal{Z}}%
 
\global\long\def\L{\mathcal{L}}%
\global\long\def\U{\mathcal{U}}%
\global\long\def\W{\mathcal{W}}%
 
\global\long\def\E{\mathcal{E}}%
\global\long\def\B{\mathcal{B}}%
\global\long\def\M{\mathcal{M}}%
 
\global\long\def\A{\mathcal{A}}%
\global\long\def\D{\mathcal{D}}%
\global\long\def\O{\mathcal{O}}%
 
\global\long\def\N{\mathcal{N}}%
 
\global\long\def\X{\mathcal{X}}%
 
\global\long\def\lm{\lim\nolimits}%
 
\global\long\def\then{\Longrightarrow}%

\global\long\def\V{\mathcal{V}}%
\global\long\def\C{\operatorname{C}}%
\global\long\def\adh{\operatorname{adh}\nolimits}%
\global\long\def\Seq{\operatorname{Seq}\nolimits}%
\global\long\def\Seqt{\operatorname{Seqt}\nolimits}%
\global\long\def\intr{\operatorname{int}\nolimits}%
\global\long\def\cl{\operatorname{cl}\nolimits}%
\global\long\def\inh{\operatorname{inh}\nolimits}%
\global\long\def\diam{\operatorname{diam}\nolimits\ }%
\global\long\def\card{\operatorname{card}}%
\global\long\def\T{\operatorname{T}}%
\global\long\def\S{\operatorname{S}}%
\global\long\def\id{\operatorname{id}}%
\global\long\def\R{\operatorname{R}}%
\global\long\def\I{\operatorname{I}}%

\global\long\def\fix{\operatorname{fix}\nolimits}%
\global\long\def\Epi{\operatorname{Epi}\nolimits}%

\global\long\def\BD{\operatorname{B}_{\mathbb{D}}}%
\global\long\def\AdhD{\operatorname{A}_{\mathbb{D}}}%
\global\long\def\UD{\operatorname{U}_{\mathbb{D}}}%
\global\long\def\K{\operatorname{K}}%
\global\long\def\g{\gimel_{\mathbb{D}}}%

\global\long\def\frak#1{\mathfrak{#1}}%
\global\long\def\cal#1{\mathcal{#1}}%
\global\long\def\surr#1{{#1}}%
\global\long\def\t#1{\mathrm{#1}}%
\global\long\def\of#1{{\left(#1\right)}}%
\global\long\def\size#1{\left\Vert #1\right\Vert }%
\global\long\def\bof#1{{\left[#1\right]}}%
\global\long\def\lsup#1#2{{}{}^{#1}{#2}}%
\global\long\def\res{{\restriction\,}}%

\global\long\def\super{\supseteq}%
\global\long\def\subnq{\subsetneqq}%
\global\long\def\all{\forall}%
\global\long\def\exi{\exists}%
\global\long\def\sub{\subseteq}%
\global\long\def\empa{\varnothing}%
\global\long\def\sem{\smallsetminus}%
\global\long\def\ity{\infty}%
\global\long\def\emp{\varnothing}%

\global\long\def\bcup{\bigcup}%
\global\long\def\bcap{\bigcap}%
\global\long\def\and{\wedge}%
\global\long\def\orr{\vee}%
\global\long\def\then{\Rightarrow}%
\global\long\def\tm{\times}%
\global\long\def\isom{\cong}%
\global\long\def\ds{\dots}%
\global\long\def\conc{^{\frown}}%
\global\long\def\nin{\notin}%
\global\long\def\mto{\mapsto}%
\global\long\def\cont{\frak c}%
\global\long\def\sym{\bigtriangleup}%
\global\long\def\clb#1{\overline{#1}}%
\global\long\def\cl{\mathrm{cl}}%
\global\long\def\clm{\mathrm{cl}_{M}}%
\global\long\def\cln{\mathrm{cl}_{N}}%
\global\long\def\clof#1{\mathrm{cl}\left(#1\right)}%

\global\long\def\brp{\mathbb{R}^{+}}%
\global\long\def\br{\mathbb{R}}%
\global\long\def\bz{\mathbb{Z}}%
\global\long\def\bq{\mathbb{Q}}%
\global\long\def\bc{\mathbb{C}}%
\global\long\def\bn{\mathbb{N}}%
\global\long\def\bg{\mathbb{G}}%
\global\long\def\ba{\mathbb{A}}%
\global\long\def\bd{\mathbb{D}}%
\global\long\def\bp{\mathbb{P}}%
\global\long\def\bu{\mathbb{U}}%
\global\long\def\bl{\mathbb{L}}%
\global\long\def\bh{\mathbb{H}}%

\global\long\def\gw{\omega}%
\global\long\def\ga{\alpha}%
\global\long\def\gb{\beta}%
\global\long\def\gd{\delta}%
\global\long\def\ph{\varphi}%
\global\long\def\gga{\gamma}%
\global\long\def\gt{\theta}%
\global\long\def\gz{\zeta}%
\global\long\def\gk{\kappa}%
\global\long\def\gs{\sigma}%
\global\long\def\gl{\lambda}%
\global\long\def\gx{\xi}%
\global\long\def\gp{\pi}%
\global\long\def\eps{\varepsilon}%
\global\long\def\th{\vartheta}%
\global\long\def\del{\Delta}%
\global\long\def\gam{\Gamma}%

\global\long\def\Ord{\text{Ord}}%
\global\long\def\dom{\text{dom}}%
\global\long\def\cof{\text{cf}}%
\global\long\def\ran{\text{ran}}%
\global\long\def\im{\text{im}}%
\global\long\def\ord{\text{On}}%

\global\long\def\abs{\left|\,\right|}%
\global\long\def\mr{\diagdown}%
\global\long\def\mc{\diagup}%
\global\long\def\sl{\cal L}%
\global\long\def\sc{\cal C}%

\global\long\def\Bs{\mathbf{\Sigma}}%
\global\long\def\Bp{\mathbf{\Pi}}%

\global\long\def\cc{\Gamma}%
\global\long\def\dd{}%
\global\long\def\rs#1{{\restriction_{#1}}}%

\global\long\def\abv#1{\left|#1\right|}%
\global\long\def\set#1{\left\{  #1\right\}  }%
\global\long\def\cur#1{\left(#1\right)}%
\global\long\def\cei#1{\left\lceil #1\right\rceil }%
\global\long\def\ang#1{\left\langle #1\right\rangle }%
 
\global\long\def\bra#1{\left[#1\right]}%

\global\long\def\se{\mathscr{E}}%
\global\long\def\sc{\mathscr{C}}%
\global\long\def\sp{\mathscr{P}}%
\global\long\def\sb{\mathscr{B}}%
\global\long\def\sa{\mathscr{A}}%
\global\long\def\ss{\mathscr{S}}%
\global\long\def\sf{\mathscr{F}}%
\global\long\def\sx{\mathscr{X}}%
\global\long\def\sh{\mathscr{H}}%
\global\long\def\su{\mathscr{U}}%
\global\long\def\so{\mathscr{O}}%
\global\long\def\sy{\mathscr{Y}}%
\global\long\def\sm{\mathscr{M}}%
\global\long\def\sn{\mathscr{N}}%
\global\long\def\sd{\mathscr{D}}%
\global\long\def\sl{\mathscr{L}}%
\global\long\def\sg{\mathscr{G}}%
\global\long\def\sk{\mathscr{K}}%
\global\long\def\sv{\mathscr{V}}%
\global\long\def\sw{\mathscr{W}}%
\global\long\def\st{\mathscr{T}}%
\global\long\def\si{\mathscr{I}}%
\global\long\def\sj{\mathscr{J}}%
\global\long\def\sr{\mathscr{R}}%

\global\long\def\df#1{\boldsymbol{#1}}%
\global\long\def\der#1{\overset{\centerdot}{#1}}%
\global\long\def\ocl#1{\left(#1\right]}%
\global\long\def\clo#1{\left[#1\right)}%
\global\long\def\iv#1{\text{Iv}\left(#1\right)}%
\global\long\def\cless#1{\text{cless}\left(#1\right)}%
\global\long\def\dens#1{\text{dens}\left(#1\right)}%
\global\long\def\fr#1#2{\frac{#1}{#2}}%
\global\long\def\mus{\mu^{*}}%
\global\long\def\srs{\sr^{*}}%
\global\long\def\rec#1{\frac{1}{#1}}%
\global\long\def\seqk#1{\left(#1\right)_{k\in\bn}}%
\global\long\def\seqi#1{\left(#1\right)_{i\in\bn}}%
\global\long\def\limk#1{\lim_{k\to\ity}#1_{k}}%
\global\long\def\back{\Leftarrow}%

\global\long\def\lisk#1{k=1,2,\dots,#1}%
\global\long\def\lisi#1{i=1,2,\dots,#1}%
\global\long\def\lisj#1{j=1,2,\dots,#1}%

\global\long\def\cupk#1{\bigcup_{k=1}^{#1}}%
\global\long\def\cupi#1{\bigcup_{i=1}^{#1}}%
\global\long\def\cupj#1{\bigcup_{j=1}^{#1}}%

\global\long\def\capk#1{\bigcap_{k=1}^{#1}}%
\global\long\def\capi#1{\bigcap_{i=1}^{#1}}%
\global\long\def\capj#1{\bigcap_{j=1}^{#1}}%

\global\long\def\seqk#1{\left(#1\right)_{k\in\bn}}%
\global\long\def\seqi#1{\left(#1\right)_{i\in\bn}}%
\global\long\def\seqj#1{\left(#1\right)_{j\in\bn}}%

\global\long\def\limk{\lim_{k\to\ity}}%
\global\long\def\limi{\lim_{i\to\ity}}%
\global\long\def\limj{\lim_{j\to\ity}}%

\global\long\def\sumi#1{\sum_{i=1}^{#1}}%
\global\long\def\sumk#1{\sum_{k=1}^{#1}}%
\global\long\def\sumj#1{\sum_{j=1}^{#1}}%

\global\long\def\qand{\qquad\text{and}\qquad}%
\global\long\def\seg{\left[0,1\right]}%
\global\long\def\bzp{\mathbb{Z}^{+}}%
\global\long\def\ovl#1{\overline{#1}}%
\global\long\def\ome{\Omega}%

\global\long\def\sfl{\mathsf{L}}%
\global\long\def\dof#1{\text{dist}\left(#1\right)}%
\global\long\def\pt{\partial}%
\global\long\def\od{\odot}%
\global\long\def\om{\ominus}%
\global\long\def\nm#1{\left\Vert #1\right\Vert }%

\global\long\def\part#1#2#3{\left\{  \left(#1{}_{1},#2{}_{1}\right),\ds,\left(#1_{#3},#2_{#3}\right)\right\}  }%
\global\long\def\parta#1{\left\{  \left(A_{1},x{}_{1}\right),\ds,\left(A_{#1},x_{#1}\right)\right\}  }%
\global\long\def\partb#1{\left\{  \left(B_{1},y{}_{1}\right),\ds,\left(B_{#1},y_{#1}\right)\right\}  }%

\global\long\def\ac{\text{AC}}%
\global\long\def\gas#1{\alpha^{*}\of{#1}}%
\global\long\def\scivp{\text{SCIVP}}%
\global\long\def\es{\text{ES}}%
\global\long\def\szy{\text{SZ}}%
\global\long\def\conn{\text{Conn}}%
\global\long\def\cn{\mbox{Cn}}%
\global\long\def\szl{\text{SZ}\cur{\mathbb{L}}}%
\global\long\def\won{\wonwon}%
\global\long\def\bb#1{\mathbb{#1}}%
\global\long\def\up#1{#1^{\uparrow}}%
\global\long\def\mesh{\mmesh}%
\global\long\def\adh{\text{adh}}%

\global\long\def\nsub{\nsubseteq}%

\global\long\def\won{\gimel}%

\global\long\def\gip{\gimel_{\Phi}}%
\global\long\def\yp{Y_{\Phi}}%

\begin{abstract}
We provide a characterization of classes of filters $\mathbb{D}$
for which the full subcategory $\fix\AdhD$ of $\mathsf{Conv}$ formed
by convergences determined by the adherence of filters of the class
$\mathbb{D}$ is simple in $\mathsf{Conv}$. Along the way, we also
elucidate when two classes of filters result in the same category
of adherence-determined convergences. As an application of the main
result, we show that the category of hypotopologies is not simple,
thus answering a question from \cite{Woj-Three}.
\end{abstract}

\section{Introduction}

It is well known that a topological space $X$ is Tychonoff (completely
regular and Hausdorff) if and only if $X$ can be embedded into a
product $\mathbb{R}^{I}$ of some copies of the real line $\mathbb{R}$.
Using the concept of reflectivity in category theory, we say that
the Tychonoff spaces form the epireflective hull of $\mathbb{R}$
or that the class of Tychonoff spaces contains an initially dense
member $\mathbb{R}$. We also say that the subcategory of Tychonoff
spaces is \emph{simple} in the category of all topological spaces.

In this paper, we will consider fundamental subcategories of the category
$\mathsf{Conv}$ of convergences and continuous functions. For a systematic
study of convergence spaces, see \cite{Dol-init,Dol-Royal,DoMy-found}.
A convergence is a relation between points of a set $X$ and the filters
on $X$. Each topological space $X$ induces a convergence on the
set $X$ by relating a point $x$ to all filters that converge to
$x$ (all filters that refine the neighborhood filter at $x$). Convergence
theory studies this relation in greater generality and considers the
topological convergence only as a special case. The need to study
non-topological convergences was pointed out by Gustave Choquet in
his fundamental paper \cite{Choq-conv}. It turns out that considering
topological problems in the larger setting of convergence spaces is
often illuminating, in a way that can be compared to using complex
numbers to solve a problem formulated in the reals.

The special subclasses of convergence spaces formed by pretopologies
and by pseudotopologies have long been recognized as fundamental:
they were already introduced by Choquet in the pioneering paper \cite{Choq-conv}
and the category $\mathsf{PreTop}$ of pretopological spaces and $\mathsf{PsTop}$
of pseudotopological spaces have special structural roles with respect
to $\mathsf{Top}$: in categorical terms, $\mathsf{PreTop}$ is the
extensional hull of $\mathsf{Top}$ while $\mathsf{PsTop}$ is the
quasitopos hull of $\mathsf{Top}$ (See, e.g., \cite{HerrLCSchwarz}
for details). These two categories turn out to fit in a larger useful
classification introduced by Szymon Dolecki: \emph{adherence-determined
convergences} \cite{Dol-quotient}. Beyond its categorical unifying
power, the notion turns out to be the key to convergence-theoretic
characterizations of various types of quotient maps (e.g., biquotient,
countably biquotient, hereditarily quotient) and of various topological
notions (e.g., bisequential, countably bisequential, Fréchet-Urysohn,
etc) in terms of functorial inequalities, which in turn allows to
unify, generalize, and refine a large spectrum of topological results
on preservation under maps \cite{Dol-quotient}, under products (e.g.,
\cite{DoMy-mech}), on compactness and its variants e.g, \cite{Dol-compact,Dol-erratum,Dol-init,myn-compact},
etc. The categories $\mathsf{ParaTop}$ and $\mathsf{HypoTop}$, of
paratopologies and of hypotopologies respectively, appear naturally
within this classification and fill important gaps left by the traditional
categorical approach in the quest to interpret classical topological
notions categorically. 

We are going to investigate the question of simplicity of subcategories
of $\mathsf{Conv}$. In particular, the category $\mathsf{Top}$ of
topological spaces is simple in $\mathsf{Conv}$. Several subcategories
of $\mathsf{Top}$ are also simple in $\mathsf{Conv}$, see for example
\cite{HaMy-On_non-symp,HaMy-The_non_symp,Herr-TopRefCoref,Herr-CatTop}.
Simplicity also holds when we enlarge $\mathsf{Top}$ to the adherence-determined
subcategories $\mathsf{PreTop}$ and $\mathsf{ParaTop}$ (see Antoine
\cite{Ant-etude} and Bourdaud \cite{Bour-espaces,Bour-cart} for
$\mathsf{PreTop}$ and \cite{Woj-Three} for $\mathsf{ParaTop}$).
However, when we enlarge $\mathsf{PreTop}$ to $\mathsf{PsTop}$ or
to the whole $\mathsf{Conv}$, then Eva and Robert Lowen showed \cite{LoLo-nonsimp}
that simplicity fails. The category $\mathsf{HypoTop}$  was shown
in \cite{Woj-Three} not to be simple, under the assumption that measurable
cardinals form a proper class. 

In this paper, we give a complete characterization of when one of
the fundamental categories of adherence-determined convergences is
simple. Namely, we give a condition on a class $\mathbb{D}$ of filters
that characterizes the simplicity of the full subcategory $\fix\AdhD$
of $\mathsf{Conv}$ formed by convergences determined by the adherence
of filters of the class $\mathbb{D}$. When $\mathbb{D}$ is respectively
the class of principal filters, of countably-based filters, of countably
complete filters, and of all filters then $\fix\AdhD$ is respectively
$\mathsf{PreTop}$, $\mathsf{ParaTop}$, $\mathsf{HypoTop}$ and $\mathsf{PsTop}$.
As a result, we recover results of \cite{Ant-etude,Bour-cart,Bour-espaces}
on the simplicity of $\mathsf{PreTop}$, of \cite{Woj-Three} on the
simplicity of $\mathsf{ParaTop}$ and of \cite{LoLo-nonsimp} on the
non-simplicity of $\mathsf{PsTop}$. Moreover, we answer the question
raised in \cite{Woj-Three} and prove in ZFC that $\mathsf{HypoTop}$
(and more generally the category of $\mu$-hypotopologies where $\mu$
is an infinite cardinal) is not a simple subcategory of $\mathsf{Conv}$
(Corollary \ref{Cor-hypo}). To summarize\medskip{}

\begin{center}
\begin{tabular}{|c|c|c|c|}
\hline 
class $\mathbb{D}$ of filters  & $\mathbb{D}$-adherence-determined conv. & category  & simple\tabularnewline
\hline 
\hline 
$\mathbb{F}$ of all & pseudotopology & $\mathsf{PsTop}$ & X\tabularnewline
\hline 
$\mathbb{F}_{1}$ of countably based & paratopology & $\mathsf{ParaTop}$ & $\checkmark$\tabularnewline
\hline 
$\mathbb{F}_{\wedge1}$ of countably complete & hypotopology & $\mathsf{HypoTop}$ & X\tabularnewline
\hline 
$\mathbb{F}_{0}$ of principal & pretopology & $\mathsf{PreTop}$ & $\checkmark$\tabularnewline
\hline 
\end{tabular}
\par\end{center}

\section{Basic definitions}

\subsection{Categorical terminology}

Let us first introduce the general concept of a simple subcategory
in general and in particular for topological categories. For more
details, we refer the reader to \cite{AdHeSt-categ,HeSt-categ,LoSiVe-CaTo}. 

Let $\cal A$ be a full and isomorphism-closed subcategory of a category
$\cal B$ with the embedding functor $E:\cal A\to\cal B$. Given an
object $B$ of $\cal B$, a pair $\cur{u,A}$, (where $A$ is an object
of $\cal A$ and $u:B\to E\of A$ is a morphism of $\cal B$) is called
an $\cal A$-\emph{reflection} of $B$ provided that for each object
$A'$ of $\cal A$ and each morphism $f:B\to E\of{A'}$ there exists
a unique morphism $g:A\to A'$ such that $f=E\of g\circ u$. The subcategory
$\cal A$ is \emph{epireflective} in $\cal B$ provided that for every
object $B$ of $\cal B$ there exists an $\cal A$-reflection $\cur{u,A}$
of $B$ with $u$ being an epimorphism of $\cal B$. We say that $\cal A$
is \emph{simple} in $\cal B$ provided that $\cal A$ is epireflective
in $\cal B$ and there exists an object $A$ of $\cal A$ such that
every epireflective (full and isomorphism-closed) subcategory of $\cal B$
containing $A$ must contain $\cal A$. We say then that $\cal A$
is the \emph{epireflective hull} of $A$ in $\cal B$.

For example, the epireflective hull of the closed interval $\bra{0,1}$
in the category of Hausdorff spaces (and continuous functions), is
the subcategory of compact Hausdorff spaces, however the epireflective
hull of the same interval in the category of all topological spaces
(and continuous functions) is the subcategory of Tychonoff spaces
(completely regular and Hausdorff). This difference is caused by the
fact that in the category of Hausdorff spaces a morphism is an epimorphism
if and only if its image is dense, while in the category of all topological
spaces being an epimorphism is equivalent to being surjective.

Let $\cal B$ be a concrete category over the category $\mathsf{Sets}$
of sets and functions, with the forgetful functor $U:\cal B\to\mathsf{Sets}$.
A class indexed family $\cur{f_{i}:B\to B_{i}}_{i\in I}$ of morphisms
of $\cal B$ is an \emph{initial source} provided that if $B'$ is
an object of $\cal B$ and $f:U\of{B'}\to U\of B$ is a function such
that $f_{i}\circ f:B'\to B_{i}$ is a morphism of $\cal B$, then
$f$ is also a morphism of $\cal B$. We say that $\cal B$ is \emph{topological}
provided that each structured source (a class indexed family $\cur{f_{i}:X\to U\of{B_{i}}}_{i\in I}$
of functions) has a unique initial lift, that is, there exists a unique
object $B$ of $\cal B$ with $U\of B=X$ such that $f_{i}:B\to B_{i}$
is a morphism of $\cal B$ for each $i\in I$ and $\cur{f_{i}:B\to B_{i}}_{i\in I}$
is an initial source in $\cal B$.

For example, the category of all topological spaces is topological,
but the category of Hausdorff spaces is not (\footnote{To see that, note that the empty structured source on a set $X$ with
at least $2$ elements has no initial lift. Otherwise, there would
be a topology on $X$ such that for any Hausdorff space $Y$ any function
$f:Y\to X$ is continuous. Such topology would have to be antidiscrete,
hence not Hausdorff.}). 

Assume that $\cal B$ is a topological category. In such a category
epimorphisms are exactly those morphisms that are surjective functions
and a (full and isomorphism-closed) subcategory $\cal A$ is epireflective
in $\cal B$ if and only if $\cal A$ is closed under the formation
of products and extremal subobjects ($\cur{f,A'}$ is an extremal
subobject of $A$ iff $f:A'\to A$ is an embedding). Moreover, for
any object $A$ of $\cal B$, there exists the epireflective hull
of $A$ in $\cal B$ obtained by taking all extremal subobjects of
the powers of $A$. An explicit condition for $\cal A$ to be simple
in $\cal B$ is that there exists an object $A_{0}$ of $\cal A$
such that for any object $A$ of $\cal A$ there exists an initial
source $\cur{f_{i}:A\to A_{0}}_{i\in I}$ in $\cal A$.

In this paper we will be concerned with simplicity of some (full and
isomorphism-closed) subcategories of the category $\mathsf{Conv}$
of convergences (and continuous functions). 

\subsection{Convergence spaces}

The context of this paper is that of the category $\mathsf{Conv}$
of convergence spaces and continuous maps. We use the terminology
and notations of \cite{DoMy-found}. In particular, a \emph{convergence
}$\xi$ on a set $X$ is a relation between points of $X$ and filters
on $X$, denoted 
\[
x\in\lm_{\xi}\F
\]
 whenever $x$ and $\F$ are $\xi$-related, subjected to two simple
axioms: $x\in\lim_{\xi}\{x\}^{\uparrow}$ for every $x\in X$, where
$\{x\}^{\uparrow}$ denotes the principal ultrafilter including $\{x\}$,
and $\lim_{\xi}\F\subset\lim_{\xi}\G$ whenever $\G$ is a filter
finer than the filter $\F$. If $(X,\xi)$ and $(Y,\tau)$ are two
convergence spaces, a map $f:X\to Y$ is \emph{continuous }(from $\xi$
to $\tau$), in symbols $f\in C(\xi,\tau)$, if 
\[
f(\lm_{\xi}\F)\subset\lm_{\tau}f[\F],
\]
where $f[\F]=\{B\subset Y:f^{-}(B)\in\F\}$ is the image filter. Of
course, every topology $\tau$ can be seen as a convergence given
by $x\in\lim_{\tau}\F$ if and only if $\F\geq\N_{\tau}(x)$, where
$\N_{\tau}(x)$ denotes the neighborhood filter of $x$ in the topology
$\tau$. This turns the category $\mathsf{Top}$ of topological spaces
and continuous maps into a full subcategory of $\mathsf{Conv}$.

We denote by $|\cdot|:\mathsf{Conv}\to\mathsf{Sets}$ the forgetful
functor, so that $|\xi|$ denotes the underlying set of a convergence
$\xi$ and $|f|$ is the underlying function of a morphism, though
we will normally not distinguish notationally the morphism and the
underlying function and denote them both by $f$. If $|\xi|=|\tau|$,
we say that $\xi$ is \emph{finer than $\tau$ }or that $\tau$ is
\emph{coarser than }$\xi$, in symbols, $\xi\geq\tau$, if the identity
map of their underlying set belongs to $C(\xi,\tau)$. This order
turns the set of convergences on a given set into a complete lattice
whose greatest element is the discrete topology, whose least element
is the antidiscrete topology, and whose suprema and infima are given
by
\begin{equation}
\lm_{\bigvee_{\xi\in\Xi}\xi}\F=\bigcap_{\xi\in\Xi}\lm_{\xi}\F\text{ and }\lm_{\bigwedge_{\xi\in\Xi}\xi}\F=\bigcup_{\xi\in\Xi}\lm_{\xi}\F.\label{eq:latticeconv}
\end{equation}

A point $x$ of a convergence space $(X,\xi)$ is \emph{isolated }if
$\{x\}^{\uparrow}$ is the only filter converging to $x$ in $\xi$.
A convergence is called \emph{prime }if it has at most one non-isolated
point.

$\mathsf{Conv}$ is a concrete topological category; in particular,
for every $f:X\to|\tau|$, there is the coarsest convergence on $X$,
called the \emph{initial convergence }for\emph{ $(f,\tau)$ }and denoted
$f^{-}\tau$, making $f$ continuous (to $\tau$), and for every $f:|\xi|\to Y$,
there is the finest convergence on $Y$, called the \emph{final convergence
for }$(f,\xi)$ and denoted $f\xi$, making $f$ continuous (from
$\xi$). Note that with these notations
\begin{equation}
f\in C(\xi,\tau)\iff\xi\geq f^{-}\tau\iff f\xi\geq\tau.\label{eq:continuity}
\end{equation}

Moreover, the initial lift on $X$ of a structured source $(f_{i}:X\to|\tau_{i}|)_{i\in I}$
turns out to be $\bigvee_{i\in I}f_{i}^{-}\tau_{i}$ and the final
lift on $Y$ of a structured sink $(f_{i}:|\xi_{i}|\to Y)_{i\in I}$
turns out to be $\bigwedge_{i\in I}f_{i}\xi_{i}$.

Products, subspaces, coproducts (sums) and quotients are then defined
as usual via initial and final structures.

Let $\Phi$ be a class of convergences. We say that a convergence
$\eta$ is \emph{initially dense} in $\Phi$ if and only if for each
$\xi\in\Phi$ there exists a set $A$ of functions from $|\xi|$ to
$|\eta|$ such that $\xi=\bigvee_{f\in A}f^{-}\eta$. Note that if
$\eta$ is initially dense in $\Phi$, then $\eta\in\Phi$ and $\xi=\bigvee_{f\in C(\xi,\eta)}f^{-}\eta$
for every $\xi\in\Phi$ (\footnote{If $\xi=\bigvee_{f\in A}f^{-}\eta$ then in particular $\xi\geq f^{-}\eta$
for every $f\in A$ so that, inview of (\ref{eq:continuity}), $A\subset C(\xi,\eta$).
Since $\xi\geq\bigvee_{f\in C(\xi,\eta)}f^{-}\eta\geq\bigvee_{f\in A}f^{-}\eta$
is always true, $\xi=\bigvee_{f\in A}f^{-}\eta$ for some $A$ if
and only if $\xi=\bigvee_{f\in C(\xi,\eta)}f^{-}\eta$.}). Note that if $\Phi$ is the class of objects of some full and isomorphism
closed subcategory $\cal A$ of $\mathsf{Conv}$, then $\eta$ is
initially dense in $\Phi$ if and only if $\cal A$ is the epireflective
hull of $\eta$ in $\mathsf{Conv}$. A class of convergences is \emph{simple}
provided it includes an initially dense convergence.

\subsection{Filters and classes of filters}

If $\mathcal{P}(X)$ denotes the powerset of $X$ and $\A\subset\mathcal{P}(X)$
then we write
\begin{eqnarray*}
\A^{\uparrow_{X}}=\A^{\uparrow} & := & \{B\subset X:\exists A\in\A,A\subset B\}\\
\A^{\#} & := & \{H\subset X:\forall A\in\A,H\cap A\neq\emptyset\}.
\end{eqnarray*}
The set $\mathbb{F}X$ of filters on $X$ is ordered by inclusion.
The infimum of a family $\mathbb{D}\subset\mathbb{F}X$ of filters
always exists and is $\bigcap_{\D\in\mathbb{D}}\D$. On the other
hand, the supremum of a pair of filters may fail to exist in $\mathbb{F}X$.
We say that two families $\A$ and $\B$ of subsets of $X$ \emph{mesh,
}in symbols\emph{ $\A\#\B$, }if $\A\subset\B^{\#}$, equivalently,
$\B\subset\A^{\#}$. Given $\F,\G\in\mathbb{F}X$ the supremum $\F\vee\G$
of the two filters exist (in $\mathbb{F}X$) if and only if $\F\#\G$.

Recall that the powerset $\mathcal{P}(X)=\{\emptyset\}^{\uparrow_{X}}$
is the \emph{degenerate filter }on $X$ and we denote by $\overline{\mathbb{F}X}$
the set of all (degenerate or proper) filters on $X$. Then inclusion
turns $\overline{\mathbb{F}X}$ into a complete lattice in which $\F\vee\G=\mathcal{P}(X)$
whenever $\F$ and $\G$ do not mesh. Note that, denoting by $\mathsf{Rel}$
the category of sets with relations as morphisms, $\overline{\mathbb{F}}:\mathsf{Rel}\to\mathsf{Rel}$
is a functor that associates with a set $X\in\mathsf{Ob(Rel)}$ the
set $\overline{\mathbb{F}X}$ and with a relation $R\subset X\times Y$
the relation $\overline{\mathbb{F}}R:\overline{\mathbb{F}X}\to\overline{\mathbb{F}Y}$
defined by 
\[
(\overline{\mathbb{F}}R)(\F)=R[\F]=\left\{ R(F):F\in\F\right\} ^{\uparrow_{Y}}.
\]
 We will denote by $\mathbb{D}\subset\overline{\mathbb{F}}$ the fact
that $\mathbb{D}$ is a subfunctor, that is, $\mathbb{D}X\subset\overline{\mathbb{F}X}$
for every set $X$ and $\overline{\mathbb{F}}R(\D)\in\mathbb{D}Y$
for every $\D\in\mathbb{D}X$ and every relation $R\subset X\times Y$.
In the terminology of \cite{JoMy-Compat,DoMy-found}, we say that
$\mathbb{D}$ is an $\mathbb{F}_{0}$-\emph{composable class of filters.
}Such a class must contain all principal filters, in particular every
principal ultrafilter. Moreover, for such a class, if $\D,\L\in\mathbb{D}X$
with $\D\#\L$ then $\D\vee\L\in\mathbb{D}X$, and if $\D\in\mathbb{D}X$
and $X\subset Y$, then $\D^{\uparrow_{Y}}\in\mathbb{D}Y$ (See e.g.,
\cite{JoMy-Compat},\cite[Lemma XIV.3.7]{DoMy-found} for this and
other properties of $\mathbb{F}_{0}$-composable classes). Among such
classes, we distinguish the class $\mathbb{F}_{0}$ of principal filters,
$\mathbb{F}_{1}$ of countably based filters and more generally $\mathbb{F}_{\kappa}$
of filters with a filter-base of cardinality less than $\aleph_{\kappa}$,
$\mathbb{F}_{\wedge\kappa}$ of $\aleph_{\kappa}$-complete filters.
In contrast, the class $\mathbb{U}$ of ultrafilters and the class
$\mathbb{E}$ of filters generated by a sequence are not $\mathbb{F}_{0}$-composable.

Given $\F\in\mathbb{F}X$ and $\mathbb{D}$ a class of filters, we
write 
\[
\mathbb{D}(\F):=\{\D\in\mathbb{D}X:\D\geq\F\}.
\]

Accordingly, $\mathbb{U}(\F)\neq\emptyset$ for every filter $\F$
and $\F=\bigcap_{\U\in\mathbb{U}(\F)}\U$ while $\F^{\#}=\bigcup_{\U\in\mathbb{U}(\F)}\U$.

Let $\xi$ be a convergence on a set $X$ and $\cal H$ be a filter
on $X$. We say that $\cal H$ \emph{adheres} to $x\in X$ (and write
$x\in\adh_{\xi}\cal H$) if there exists a filter $\cal G$ that refines
$\cal H$ with $x\in\lim_{\xi}\cal G$. In other words,
\begin{equation}
\adh_{\xi}\H:=\bigcup_{\mathbb{F}X\ni\G\geq\H}\lm_{\xi}\G=\bigcup_{\mathbb{F}X\ni\F\#\H}\lm_{\xi}\F=\bigcup_{\U\in\mathbb{U}(\H)}\lm_{\xi}\U.\label{eq:defadh}
\end{equation}

Let $\mathbb{D}$ be a class of filters. A convergence $\xi$ is $\mathbb{D}$-\emph{adherence-determined}
if $x\in\lim_{\xi}\cal F$ whenever $x\in\t{adh}_{\xi}\D$ for each
filter $\D\in\mathbb{D}$ such that $\D$ is a filter on $\abv{\xi}$
and $\D\#\cal F$.

If $\mathbb{D}$ is an $\mathbb{F}_{0}$-composable class of filters,
then $\AdhD$ defined on objects by
\[
\lm_{\AdhD\xi}\F=\bigcap_{\mathbb{D}\ni\D\#\F}\adh_{\xi}\D
\]
 is a concrete reflector and $\fix\AdhD=\{\xi\in\mathsf{Ob(Conv)}:\xi=\AdhD\xi\}$
is the subcategory of $\mathbb{D}$-adherence-determined convergences. 

\bigskip{}

\begin{center}
\begin{tabular}{|c|c|c|c|}
\hline 
$\mathbb{D}$ & $\AdhD$ & $\xi=\AdhD\xi$ is a & category $\fix\AdhD$\tabularnewline
\hline 
\hline 
$\mathbb{F}$ & $\S$ & pseudotopology & $\mathsf{PsTop}$\tabularnewline
\hline 
$\mathbb{F}_{1}$ & $\S_{1}$ & paratopology & $\mathsf{ParaTop}$\tabularnewline
\hline 
$\mathbb{F}_{\wedge1}$ & $\S_{\wedge1}$ & hypotopology & $\mathsf{HypoTop}$\tabularnewline
\hline 
$\mathbb{F}_{0}$ & $\S_{0}$ & pretopology & $\mathsf{PreTop}$\tabularnewline
\hline 
\end{tabular}
\par\end{center}

\section{Main results}

\global\long\def\AdhH{\operatorname{A}_{\mathbb{H}}}%

\subsection{For what classes $\mathbb{D}$ and $\mathbb{H}$ do we have $\protect\AdhD=\protect\AdhH$?(\protect\footnote{We would like to thank Emilio Angulo-Perkins, Fadoua Chigr, and Jesús
González Sandoval for helpful discussions around the results in this
subsection.})}

Given a class $\mathbb{D}$, define for each set $X$
\[
\widehat{\mathbb{D}}X:=\left\{ \H\in\mathbb{F}X:\forall\U\in\mathbb{U}(\H)\exists\D\in\mathbb{D}X:\H\leq\D\leq\U\right\} 
\]
thus defining a new class $\widehat{\mathbb{D}}$ (\footnote{Note that 
\[
\widehat{\mathbb{D}}X=\left\{ \H\in\mathbb{F}X:\forall\F\in\mathbb{F}X\left(\F\#\H\then\exists\D\in\mathbb{D}(\H),\D\#\F\right)\right\} .
\]
}). By definition $\mathbb{D}\subset\widehat{\mathbb{D}}$, and 
\begin{equation}
\mathbb{D}\subset\mathbb{J}\then\widehat{\mathbb{D}}\subset\widehat{\mathbb{J}}.\label{eq:hatincreasing}
\end{equation}
Moreover, 
\begin{lem}
\label{lem:hatidempotent} Given a class $\mathbb{D}$ of filters,
$\widehat{\mathbb{D}}=\widehat{\widehat{\mathbb{D}}}$. 
\end{lem}

\begin{proof}
As $\mathbb{D}\subset\widehat{\mathbb{D}}$, (\ref{eq:hatincreasing})
implies that $\widehat{\mathbb{D}}\subset\widehat{\widehat{\mathbb{D}}}$.
If $\H\in\widehat{\widehat{\mathbb{D}}}$ then for every $\U\in\mathbb{U}(\H)$
there is $\F\in\widehat{\mathbb{D}}$ with $\H\leq\F\leq\U$. As $\F\in\widehat{\mathbb{D}}$
and $\U\in\mathbb{U}(\F)$ there is $\D\in\mathbb{D}$ with $\H\leq\F\leq\D\leq\U$
and thus $\H\in\widehat{\mathbb{D}}$.
\end{proof}
\begin{thm}
\label{thm:ADgeqAH} Given two classes of filters $\mathbb{D}$ and
$\mathbb{H}$,
\[
\AdhD\geq\AdhH\iff\mathbb{H}\subset\widehat{\mathbb{D}}.
\]
\end{thm}

\begin{proof}
Assume that $\mathbb{H}\subset\widehat{\mathbb{D}}$ and let $x\in\lim_{\AdhD\xi}\F$
and $\H\in\mathbb{H}$ with $\H\#\F$. Because $\H\in\widehat{\mathbb{D}}$,
there is $\D\in\mathbb{D}(\H)$ with $\D\#\F$. As $\D\in\mathbb{D}$
and $\D\#\F$, $x\in\adh_{\xi}\D$. As $\D\geq\H$, $\adh_{\xi}\D\subset\adh_{\xi}\H$.
Hence $x\in\lim_{\AdhH\xi}\F$.

Assume conversely that $\mathbb{H}\not\subset\widehat{\mathbb{D}}$,
that is, there is $\H_{0}\in\mathbb{H}$ with $\H_{0}\notin\widehat{\mathbb{D}}$,
so that there is $\U_{0}\in\mathbb{U}(\H_{0})$ such that $\H_{0}\nleq\D$
whenever $\D\in\mathbb{D}$ and $\D\leq\U_{0}$. Consider the prime
convergence $\sigma$ on $X\cup\{\infty\}$ in which $\infty\in\lim_{\sigma}\F$
if and only if $\F$ and $\H_{0}$ do not mesh. Then $\sigma=\AdhH\sigma$
because $\H_{0}\in\mathbb{H}$, but $\AdhD\sigma\ngeq\sigma$. Indeed,
by definition $\infty\notin\lim_{\sigma}\U_{0}$ because $\H_{0}\#\U_{0}$
but $\infty\in\lim_{\AdhD\sigma}\U_{0}$. Indeed, if $\D\in\mathbb{D}$
and $\D\#\U_{0}$, equivalently, $\D\leq\U_{0}$ then $\H_{0}\nleq\D$,
that is, there is $H\in\H_{0}$ with $H^{c}\in\D^{\#}$. As $\{H^{c}\}^{\uparrow}$
does not mesh with $\H_{0}$, $\infty\in\lim_{\sigma}\{H^{c}\}^{\uparrow}$
and thus $\infty\in\adh_{\sigma}\D$, which completes the proof that
$\AdhD\sigma\ngeq\sigma$. As a result, $\AdhD\ngeq\AdhH$.
\end{proof}
\begin{cor}
Given two classes of filters $\mathbb{D}$ and $\mathbb{H}$,
\[
\AdhD=\AdhH\iff\widehat{\mathbb{H}}=\widehat{\mathbb{D}}.
\]
\end{cor}

\begin{proof}
Assume $\AdhD=\AdhH$. In view of Theorem \ref{thm:ADgeqAH}, $\mathbb{H}\subset\widehat{\mathbb{D}}$
and $\mathbb{D}\subset\widehat{\mathbb{H}}$. In view of (\ref{eq:hatincreasing})
and Lemma \ref{lem:hatidempotent}, $\widehat{\mathbb{H}}\subset\widehat{\widehat{\mathbb{D}}}=\widehat{\mathbb{D}}$
and $\widehat{\mathbb{D}}\subset\widehat{\widehat{\mathbb{H}}}=\widehat{\mathbb{H}}$
so that $\widehat{\mathbb{H}}=\widehat{\mathbb{D}}$. 

Conversely, if $\widehat{\mathbb{H}}=\widehat{\mathbb{D}}$ then $\mathbb{D}\subset\widehat{\mathbb{D}}=\widehat{\mathbb{H}}$
and $\mathbb{H}\subset\widehat{\mathbb{H}}=\widehat{\mathbb{D}}$
so that $\AdhD=\AdhH$ by Theorem \ref{thm:ADgeqAH}.
\end{proof}
\begin{example}
Of course, $\widehat{\mathbb{U}}=\mathbb{F}=\mathbb{\widehat{F}}$.
Moreover, $\widehat{\mathbb{F}_{0}}=\text{\ensuremath{\mathbb{F}_{0}}}$.
To see the latter, assume that $\H\notin\mathbb{F}_{0}$, so that
$(\ker\H)^{c}\in\H^{\#}$. Then there is $\U\in\mathbb{U}(\H\vee(\ker\H)^{c})$.
Note that any $\D\in\mathbb{F}_{0}(\H)$ is of the form $\{D\}^{\uparrow}$
for $D\subset\ker\H$, so that $D\notin\U$. Hence $\H\notin\widehat{\mathbb{F}_{0}}$.
\end{example}

\begin{example}
A topological (or convergence) space $(X,\xi)$ is called \emph{bisequential
}if for every ultrafilter $\U$ with $x\in\lim_{\xi}\U$ there is
$\D\in\mathbb{F}_{1}$ with $\D\leq\U$ and $x\in\lim_{\xi}\D$. In
other words, $\xi=\T\xi$ is bisequential if and only if each neighborhood
filter $\N_{\xi}(x)$ is in $\widehat{\mathbb{F}_{1}}$. Naturally,
we call filters of $\widehat{\mathbb{F}_{1}}$ \emph{bisequential
filters. }As there are bisequential topological spaces that are not
first-countable (\footnote{Take for instance the one-point compactification of a discrete set
$X$ of cardinality that is not measurable. See \cite[Example 10.15]{michael}
for details.}), $\mathbb{F}_{1}$ is a proper subclass of $\widehat{\mathbb{F}_{1}}$
but 
\[
\operatorname{A}_{\mathbb{F}_{1}}=\operatorname{A}_{\widehat{\mathbb{F}_{1}}}=\S_{1}.
\]
\end{example}

\subsection{For what class $\mathbb{D}$ is $\protect\fix\protect\AdhD$ simple?}
\begin{defn}
\label{def:refinable class} A class $\mathbb{D}$ of filters is called
\emph{refinable }if there is a set $Y$ such that for any set $X$
and every $\D\in\mathbb{D}X$ and $\F\in\mathbb{F}X$ with $\F\#\D$
there is $\L\in\mathbb{D}(\D)$ with $\L\#\F$ and there is $f:X\to Y$
with $\D\leq f^{-}[f[\L]]$.
\end{defn}

\begin{lem}
\label{lem:notmesh} If $\D\leq f^{-}[f[\L]]$ then 
\[
\H\neg\#\D\then f[\H]\neg\#f[\L].
\]
\end{lem}

\begin{proof}
If $f[\H]\#f[\L]$, equivalently, $\H\#f^{-}[f[\L]]$ then, in particular,
$\H\#\D$ because $\D\leq f^{-}[f[\L]]$. 
\end{proof}
If $\mathbb{D}$ is refinable, given $Y$ as in Definition \ref{def:refinable class},
define the convergence space $\gimel_{\mathbb{D}}$ defined on 
\[
Y_{\infty}=Y\cup\{\infty_{0}\}\cup\{y_{\G}:\G\in\mathbb{D}Y\},
\]
 by 
\[
Y\cup\{\infty_{0}\}\subset\lm_{\gimel_{\mathbb{D}}}\F
\]
 for every $\F\in\mathbb{F}(Y_{\infty})$ and $y_{\G}\in\lim_{\gimel_{\mathbb{D}}}\F$
if $\F$ and $\G^{\uparrow}\wedge\{\infty_{0}\}^{\uparrow}$ do not
mesh, where $\G{}^{\uparrow}$ is the filter generated on $Y_{\infty}$
by $\G$. Note that, by definition, for every $\G\in\mathbb{D}Y$,
\begin{equation}
y_{\G}\notin\adh_{\g}\G^{\uparrow}.\label{eq:adhD}
\end{equation}

\begin{thm}
\label{cor:characterizationD} Let $\mathbb{D}$ be an $\mathbb{F}_{0}$-composable
class of filters. The category $\fix\AdhD$ is simple (in $\mathsf{Conv}$)
if and only if $\mathbb{D}$ is refinable.
\end{thm}

\begin{proof}
Assume that $\mathbb{D}$ is not refinable and let $(Y,\tau_{0})$
with $\tau_{0}=\AdhD\tau_{0}$. Because $\mathbb{D}$ is not refinable,
there is $X$, $\D_{0}\in\mathbb{D}X$ and $\F_{0}\in\mathbb{F}X$
with $\D_{0}\#\F_{0}$ and for every $\L\in\mathbb{D}(\D_{0})$ with
$\L\#\F_{0}$ and every $f\in Y^{X}$, we have $\D\nleq f^{-}[f[\L]]$,
that is, there is $D_{\L,f}\in\D_{0}$ with $D_{\L,f}\notin f^{-}[f[\L]]$,
that is, $D_{\L,f}^{c}:=X\setminus D_{\L,f}$ belongs to $\left(f^{-}[f[\L]]\right)^{\#}$,
equivalently, 
\begin{equation}
\exists D_{\L,f}\in\D_{0}:f(D_{\L_{,}f}^{c})\in(f[\L])^{\#}.\label{eq:Ds}
\end{equation}

Let $\xi$ be the prime convergence on $X\cup\{\infty\}$ defined
by $\infty\in\lim_{\xi}\F$ if and only if $\F$ and $\D_{0}$ do
not mesh. Note that by definition $\infty\notin\adh_{\xi}\D_{0}$.
This is easily seen to be a convergence. Moreover, $\xi=\AdhD\xi$
because if $\infty\in\adh_{\xi}\L$ for every $\L\in\mathbb{D}X$
with $\L\#\F$ then $\D_{0}\neg\#\F$ (for otherwise $\infty\in\adh_{\xi}\D_{0}$
because $\D_{0}\in\mathbb{D}X$). 

In particular, $\infty\notin\lim_{\xi}\F_{0}$. We will see that $\infty\in\lim_{\bigvee_{f\in C(\xi,\tau_{0})}f^{-}\tau_{0}}\F_{0}$
so that $\xi\neq\bigvee_{f\in C(\xi,\tau_{0})}f^{-}\tau_{0}$ and
as a result $\fix\AdhD$ is not simple.

To see this, let $f\in C(\xi,\tau_{0})$ and $\G\in\mathbb{D}Y$ with
$\G\#f[\F_{0}]$. Then $f^{-}[\G]\#\F_{0}$. Either $f^{-}[\G]$ and
$\D_{0}$ do not mesh or they do. In the former case, $\infty\in\lim_{\xi}f^{-}[\G]$
and by continuity, $f(\infty)\in\lim_{\tau_{0}}f[f^{-}[\G]]$. As
$f[f^{-}[\G]]\#\G$, $\infty\in\adh_{\tau_{0}}\G$. In the later case,
the filter $\L:=f^{-}[\G]\vee\D_{0}$ belongs to $\mathbb{D}X$ because
$\mathbb{D}$ is $\mathbb{F}_{0}$-composable. By (\ref{eq:Ds}) and
continuity of $f$, $f(\infty)\in\adh_{\tau_{0}}f[\L]$ because $\infty\in\lim_{\xi}\{D_{\L,f}^{c}\}^{\uparrow}$.
Moreover, $f[\L]\geq f[f^{-}[\G]]\geq\G$ so that $f(\infty)\in\adh_{\xi}\G$.
Hence $f(\infty)\in\lim_{\tau_{0}}f[\F_{0}]$ for every $f\in C(\xi,\tau_{0})$,
that is, $\infty\in\lim_{\bigvee_{f\in C(\xi,\tau_{0})}f^{-}\tau_{0}}\F_{0}$.

Assume now that $\mathbb{D}$ is refinable. We will show that $\gimel_{\mathbb{D}}$
is initially dense in $\fix\AdhD$.

We first check that $\gimel_{\mathbb{D}}=\AdhD\gimel_{\mathbb{D}}$.
If $y\in Y_{\infty}\setminus\lim_{\g}\F$ then $y=y_{\D}$ for some
$\D\in\mathbb{D}Y$ and, and $\F\#(\D^{\uparrow}\wedge\{\infty_{0}\}^{\uparrow})$
so that $\infty_{0}\in\bigcap_{F\in\F,F\notin\D^{\#}}F$. If $\infty_{0}\in\ker\F$
then $\{\infty_{0}\}^{\uparrow}\in\mathbb{D}Y_{\infty}$, $\{\infty_{0}\}^{\uparrow}\#\F$
and $y_{\D}\notin\adh_{\g}\{\infty_{0}\}^{\uparrow}$. Else, there
is $F_{0}\in\F$ with $\infty_{0}\notin F_{0}$. Since $\F\#(\D^{\uparrow}\wedge\{\infty_{0}\}^{\uparrow})$
Then $\D^{\uparrow}\in\mathbb{D}Y_{\infty}$, $\D^{\uparrow}\#\F$
and $y_{\D}\notin\adh_{\g}\D^{\uparrow}$. 

To see that $\g$ is initially dense in $\fix\AdhD$, consider $\xi=\AdhD\xi$
on $X$ and suppose that $x\notin\lim_{\xi}\F$, so that there is
$\D\in\mathbb{D}X$ with $\F\#\D$ and $x\notin\adh_{\xi}\D$. By
refinability, there is $\L\in\mathbb{D}(\D)$ (so that $x\notin\adh_{\xi}\L$)
with $\L\#\F$ and there is $f_{0}:X\to Y$ with $\D\leq f_{0}^{-}[f_{0}[\L]]$. 

Let $h:X\to Y_{\infty}$ be defined by $h(t)=f_{0}(t)$ for $t\notin\{x\}\cup\ker\D$,
$h(x)=y_{\G}$ for $\G:=f_{0}[\L]\in\mathbb{D}Y$ and $h(t)=\infty_{0}$
for $t\in\ker\D$. We show that $h\in C(\xi,\gimel_{\mathbb{D}})$
and $h(x)\notin\lim_{\gimel_{\mathbb{D}}}h[\F]$ so that $\xi\leq\bigvee_{h\in C(\xi,\gimel_{\mathbb{D}})}h^{-}\gimel_{\mathbb{D}}$.

To see that $h\in C(\xi,\gimel_{\mathbb{D}})$, note that if $t\in\lim_{\xi}\H$
and $t\neq x$ then 
\[
h(t)\in Y\cup\{\infty_{0}\}\subset\lm_{\gimel_{\mathbb{D}}}h[\H],
\]
hence we only need to consider the case $x\in\lim_{\xi}\H$. Then
$\D$ and $\H$ do not mesh because $x\notin\adh_{\ensuremath{\xi}}\D$.
In particular $\ker\D\notin\H^{\#}$. Moreover $h(x)=y_{\G}$. If
$x\notin\ker\H$, $h[\H]=f_{0}[\H]$. In view of Lemma \ref{lem:notmesh},
$f_{0}[\H]$ and $\G$ do not mesh so $y_{\G}=h(x)\in\lim_{\gimel_{\mathbb{D}}}h[\H]$.

To see that $h(x)\notin\lim_{\gimel_{\mathbb{D}}}h[\F]$ note that
as $\F\#\L$ then $f_{0}[\F]\#\G$. Moreover, if $\ker\D\in\F$ then
$h[\F]=\{\infty_{0}\}^{\uparrow}$ does not converge to $y_{\G}$.
Else $(\ker\D)^{c}\in\F^{\#}$ and $h[\F]=f_{0}[\F]$ meshes with
$\G$ so $h(x)\notin\lim_{\gimel_{\mathbb{D}}}h[\F]$.
\end{proof}
A class $\mathbb{D}$ of filters is called \emph{fiber-stable }if
there is a set $Y$ such that for every set $X$ and every $\D\in\mathbb{D}X$
there is $f:X\to Y$ with $\D\leq f^{-}[f[\D]]$. Of course, every
fiber-stable class is also refinable, as we can then take $\L=\D$
in the definition of a refinable class. Hence, it is sufficient for
$\fix\AdhD$ to be simple that the class $\mathbb{D}$ be fiber-stable. 
\begin{example}
\label{exa:PrTopsimple} The category $\mathsf{PreTop}$ of pretopologies
is $\fix\operatorname{A}_{\mathbb{F}_{0}}$ and is simple because
$\mathbb{F}_{0}$ is fiber-stable, hence refinable. Indeed, taking
$Y=\{0,1\}$, then for every $X$ and $\{A\}^{\uparrow}\in\mathbb{F}X$,
$\{A\}^{\uparrow}\leq f^{-}[f[\{A\}^{\uparrow}]]$ where $f(x)=1$
if and only if $x\in A$. That $\mathsf{Prtop}$ is simple is known
from \cite[II.2]{Bour-espaces}. Note that $\gimel_{\mathbb{F}_{0}}$
is an initially dense object of $\mathsf{Prtop}$ that is different
from the Bourdaud pretopology on 3 points. Indeed, it has 6 points. 
\end{example}

\begin{example}
\label{exa:ParTopsimple} The category $\mathsf{ParaTop}$ of paratopologies
is $\fix\operatorname{A}_{\mathbb{F}_{1}}$ and is simple because
$\mathbb{F}_{1}$ is fiber-stable, hence refinable. Take $Y=\omega$.
Given $X$ and $\D\in\mathbb{F}_{1}X$, we can deal with $\D$ with
a two-valued map as in Example \ref{exa:PrTopsimple} if $\D$ is
principal. Otherwise, $\D\vee(\ker\D)^{c}$ is a non-degenerate free
countably based filter and thus has a decreasing filter base $(H_{n})_{n\in\omega}$
with $H_{1}=X\setminus\ker\D$. Consider $f:X\to Y$ defined by $f(x)=1$
if $x\in\ker\D$ and $f(x)=n>1$ if $x\in H_{n-1}\setminus H_{n}$.
Then $f^{-}[f[\D]]\geq\D$. That $\mathsf{Partop}$ is simple is \cite[Theorem 1]{Woj-Three}
and $\gimel_{\mathbb{F}_{1}}$ is a slight simplification of the initially
dense object $\gimel$ used in \cite{Woj-Three}.
\end{example}

Though fiber-stability is often more practical to check, there are
refinable classes that are not fiber-stable:
\begin{example}
The class $\widehat{\mathbb{F}_{1}}$ is refinable but not fiber-stable.
$\widehat{\mathbb{F}_{1}}$ is refinable because $\operatorname{A}_{\mathbb{F}_{1}}=\operatorname{A}_{\widehat{\mathbb{F}_{1}}}$
is simple. To see that $\widehat{\mathbb{F}_{1}}$ is not fiber-stable,
given any set $Y$ let $X$ be a set of non-measurable cardinality
$\card X>\card Y$. The cofinite filter $\H$ on $X$ is then a bisequential
filter (See \cite[Example 10.15]{michael}), that is, $\H\in\widehat{\mathbb{F}_{1}}$.
On the other hand, there is $y\in Y$ with an infinite fiber $A=f^{-}(y)$,
so that $A\#\H$ and thus $f(A)=\{y\}\in(f[\H])^{\#}$. Hence, if
$x\in A$ then $X\setminus\{x\}\in\H$ but $X\setminus\{x\}\notin f^{-}[f[\H]]$.
\end{example}

\section{Non-simplicity of the class of $\mu$-hypotopologies}

For background on set theory we refer the reader to \cite{Jech}.
Let $\mu$ be an infinite cardinal. A filter $\cal H$ is $\mu$-\emph{complete}
provided $\bigcap\cal H'\in\cal H$ for every $\cal H'\sub\cal H$
with $\abv{\cal H'}<\mu$. Note that each filter is $\aleph_{0}$-complete.
A convergence $\xi$ on $X$ is a $\mu$-\emph{hypotopology} iff it
is $\bh$-adherence-determined, where $\bh$ is the class of all $\mu$-complete
filters. In particular, a convergence $\xi$ is a \emph{pseudotopology}
if and only if it is an $\aleph_{0}$-hypotopology (is $\mathbb{F}$-adherence-determined
with $\mathbb{F}$ being the class of all filters) and $\xi$ is a
\emph{hypotopology} if and only if it is an $\aleph_{1}$-hypotopology
(is $\bh$-adherence-determined for $\bh$ consisting of all countably
complete filters).

Let $\gl$ is a regular uncountable cardinal and $A\sub\gl$. We say
that $A$ is \emph{unbounded} in $\gl$ if there are no upper bound
on $A$ in $\gl$ and we say that $A$ is closed in $\gl$ if $\sup A'\in A$
for any $A'$ that is bounded in $\gl$ (this is equivalent to $A$
being closed in the order topology on $\gl$). The closed unbounded
subsets of $\gl$ form a filter base and the filter on $\gl$ generated
by them is called the \emph{closed unbounded filter} on $\gl$. This
filter is $\gl$-complete.
\begin{lem}
\label{lem4} Let $Y$ be a set, $\gl>\t{card}\,Y$ be an uncountable
regular cardinal and $\cal C$ be the closed unbounded filter on $\gl$.
Then for each $f:\gl\to Y$ there exists a uniform $\gl$-complete
filter $\cal F_{f}$ on $\gl$ such that $f\bof{\cal F_{f}}=f\bof{\cal C}$
and $\cal F_{f}$ does not mesh with $\cal C$.
\end{lem}

\begin{proof}
Let $f:\gl\to Y$ be arbitrary. Define $\cal P:=\set{f^{-}\of y:y\in f\bof{\gl}}$
to be the family of fibers of $f$ with 
\[
\cal P_{0}:=\set{P\in\cal P:\t{card}\,P<\gl}
\]
and $\cal P_{1}:=\cal P\sem\cal P_{0}$. Note that the regularity
of $\gl$ implies that $\t{card}\,P_{0}<\gl$, where
\[
P_{0}:=\bigcup_{P\in\cal P_{0}}P
\]
and so $\cal P_{1}$ is not empty. We claim that there exists $C\in\cal C$
such that both $P\cap C$ and $P\sem C$ have cardinality $\gl$ for
every $P\in\cal P_{1}$. Let $\cal P_{1}$ be enumerated as $\set{P_{\xi}:\xi<\gk}$
for some cardinal $\gk\le\t{card}\,Y$.

We will use transfinite induction to construct two sequences $\cur{A_{\ga}}_{\ga<\gl}$
and $\cur{B_{\ga}}_{\ga<\gl}$ of subsets of $\gl$ such that
\begin{itemize}
\item any two distinct members of the family $\set{A_{\ga}:\ga<\gl}\cup\set{B_{\ga}:\ga<\gl}$
are disjoint.
\item for any $\ga<\gl$ we have $A_{\ga}=\set{\gga_{\ga,\xi}:\xi<\gk}$
and $B_{\ga}=\set{\gd_{\ga,\xi}:\xi<\gk}$ with $\gga_{\ga,\xi},\gd_{\ga,\xi}\in P_{\xi}$
for every $\xi<\gk$
\item $\clof{\bigcup_{\ga<\gl}A_{\ga}}\cap\bigcup_{\ga<\gl}B_{\ga}=\emp$,
where $\cl$ is the closure operation in the order topology on $\gl$.
\end{itemize}
Taking $C:=\clof{\bigcup_{\ga<\gl}A_{\ga}}$ satisfies the requirements.

Suppose that $\gb<\gl$ is an ordinal such that $A_{\ga}$ and $B_{\ga}$
are defined for each $\ga<\gb$ and that
\begin{itemize}
\item any two distinct members of the family $\set{A_{\ga}:\ga<\gb}\cup\set{B_{\ga}:\ga<\gb}$
are disjoint.
\item for any $\ga<\gb$ we have $A_{\ga}=\set{\gga_{\ga,\xi}:\xi<\gk}$
and $B_{\ga}=\set{\gd_{\ga,\xi}:\xi<\gk}$ with $\gga_{\ga,\xi},\gd_{\ga,\xi}\in P_{\xi}$
for every $\xi<\gk$
\item $\clof{\bigcup_{\ga<\gb}A_{\ga}}\cap\bigcup_{\ga<\gb}B_{\ga}=\emp$.
\end{itemize}
For each $\xi<\gk$, the set
\[
P_{\xi}':=\set{\gga_{\ga,\xi}:\ga<\gb}\cup\set{\gd_{\ga,\xi}:\ga<\gb}
\]
is a subset of $P_{\xi}$ of cardinality $<\gl$ so there is $\gga_{\gb,\xi}\in P_{\xi}$
with $\gga_{\gb,\xi}>\sup P_{\xi}'$. Let $A_{\gb}:=\set{\gga_{\gb,\xi}:\xi<\gk}$. 

For each $\xi<\gk$, let $\gd_{\gb,\xi}\in P_{\xi}$ be such that
$\gd_{\gb,\xi}>\sup A_{\gb}.$ Let $B_{\gb}:=\set{\gd_{\gb,\xi}:\xi<\gb}$.
It is clear that the obtained sequences $\cur{A_{\ga}}_{\ga<\gl}$
and $\cur{B_{\ga}}_{\ga<\gl}$ satisfy the requirements.

Let $C\in\cal C$ be such that both $P\cap C$ and $P\sem C$ have
cardinality $\gl$ for every $P\in\cal P_{1}$.

Since $\t{card}\,P_{0}<\gl$, it follows that $P_{1}:=\bigcup_{P\in\cal P_{1}}\in\cal C$
so $P_{1}\cap C\in\cal C$. Let $h:P_{1}\cap C\to P_{1}\sem C$ be
a bijection such that if $x\in P\cap C$ for some $P\in\cal P_{1}$,
then $h\of x\in P\sem C$. Extend $h$ to a bijection $h:\gl\to\gl$
by declaring that $h\of x:=x$ whenever $x\in P_{0}$. Let $\cal F_{f}:=h\bof{\cal C}$.
Since $h$ is a bijection and $\cal C$ is a uniform $\gl$-complete
filter on $\gl$, it follows that $\cal F_{f}$ is uniform and $\gl$-complete.
It is clear that $f\bof{\cal F_{f}}=f\bof{\cal C}$. Since $P_{1}\cap C\in\cal C$
and $P_{1}\sem C\in\cal F_{f}$, it follows that $\cal F_{f}$ does
not mesh with $\cal C$.
\end{proof}
\begin{thm}
\label{Thm-hypo} For any infinite cardinal $\mu$ the class of $\mu$-hypotopologies
is not simple.
\end{thm}

\begin{proof}
We show that the class $\mathbb{D}$ of $\mu$-complete filters is
not refinable, that is, for every $Y$ there is $X$ (with $X=\lambda\geq\mu$
a regular uncountable cardinal), and there is $\D\in\mathbb{D}X$
and $\F\#\D$ (taking $\D=\F=\mathcal{C}$ the closed unbounded filter
on $\gl$) such that for every $\L\in\mathbb{D}(\D)$ with $\L\#\F$
and every $f:X\to Y$, $\D\nleq f^{-}[f[\L]]$. Indeed, if $\L\geq\mathcal{D}$
then for every $f:X\to Y$, take $\F_{f}$ as in Lemma \ref{lem4}
to the effect that$f[\L]\geq f[\D]=f[\F_{f}]$ and thus $f[\L]\#f[\F_{f}]$,
equivalently, $\F_{f}\#f^{-}[f[\L]]$. As $\F_{f}$ does not mesh
with $\D$, $f^{-}[f[\L]]\ngeq\D$. The conclusion follows from Theorem
\ref{cor:characterizationD}.
\end{proof}
\begin{rem}
Note that Lemma \ref{lem4} is the key to Theorem \ref{Thm-hypo}
and a direct proof based on this lemma rather than through Theorem
\ref{cor:characterizationD} is relatively easy: for $\lambda$ and
$\mathcal{C}$ as in Lemma \ref{lem4}, let $\xi$ be a convergence
on $\gl$ defined by $\ga\in\lim_{\xi}\cal F$ iff $\cal F=\up{\set{\ga}}$
for $\ga>0$ and $0\in\lim_{\xi}\cal F$ iff $\bigcap\cal F\sub\set 0$
and $\cal F$ does not mesh with $\cal C$. This convergence can be
shown to be a $\mu$-hypotopology. Now, for every convergence space
$(Y,\tau)$ and infinite $\mu$, pick $\lambda\geq\mu$ uncountable
and non-measurable. It is not difficult to verify, using Lemma \ref{lem4},
that the corresponding convergence $\xi$ satisfies $\xi\neq\bigvee_{f\in C(\xi,\tau)}f^{-}\tau$.
\end{rem}

The following answers affirmatively \cite[Problem 3]{Woj-Three}.
\begin{cor}
\label{Cor-hypo} The class of hypotopologies is not simple.
\end{cor}

\begin{proof}
Hypotopologies are $\aleph_{1}$-hypotopologies.
\end{proof}
Moreover, we recover the main result of \cite{LoLo-nonsimp}.
\begin{cor}
The class of pseudotopologies is not simple.
\end{cor}

\begin{proof}
Pseudotopologies are $\aleph_{0}$-hypotopologies.
\end{proof}

\end{document}